\documentclass[oneside,11pt]{amsart}

\makeatletter
\@namedef{subjclassname@2020}{\textup{2020} Mathematics Subject Classification}
\makeatother

\usepackage[a4paper,width=140mm,top=27mm,bottom=27mm]{geometry}
\usepackage{amsmath,amssymb,mathtools}
\usepackage[hidelinks]{hyperref}
\usepackage{xcolor}
\usepackage{enumitem}
\usepackage{mathrsfs}

\newtheorem{theorem}{Theorem}[section]
\newtheorem{lemma}[theorem]{Lemma}

\theoremstyle{remark}
\newtheorem{remark}[theorem]{Remark}

\newcommand{\R}{\mathbb R}
\newcommand{\T}{\mathbb T}

\newcommand{\dd}{\,\mathrm d}
\newcommand{\whM}{\widehat M}
\newcommand{\whK}{\widehat K}

\newcommand{\norm}[2][]{\left\|#2\right\|_{#1}}

\numberwithin{equation}{section}

\begin{document}
\address{Yongming Luo
\newline \indent
Faculty of Computational Mathematics and Cybernetics
\newline \indent Shenzhen MSU-BIT University, China}
\email{luo.yongming@smbu.edu.cn}

\title[Sharp mass-threshold for Dancer-type solutions]
{Sharp mass-threshold for Dancer-type solutions of the focusing mass-critical NLS on $\R^d\times\T$}
\author{Yongming Luo}

\begin{abstract}
The mass-critical NLS on Euclidean space $\R^d$ exhibits a strong mass rigidity: all positive ground states are generated from a single profile and have the same ground state mass $\widehat{M}(Q)$.  By appealing to bifurcation methods, Dancer constructed in his seminar paper \cite{DancerSolution} solutions to the corresponding equation on $\R^d\times\T$ which decay in the noncompact directions and are nontrivially periodic in one direction.  Such bifurcation approach, however, does not provide any energetic characterization of the solutions, and in particular does not explain their relation to the Euclidean ground-states.  By introducing a new strict monotonicity mechanism for the prescribed-mass energy level, combining the semivirial-vanishing geometry framework developed in author's recent work, we prove that for any mass $c\in(0,2\pi\widehat{M}(Q))$ the semivirial-vanishing variational problem $m_c$ admits a normalized Dancer-type optimizer which also solves the focusing mass-critical NLS on $\R^d\times\T$. This also gives a sharp complement for the existence results deduced in our earlier work \cite{Luo_LegendreFenchel} via the Legendre-Fenchel duality.
\end{abstract}


\maketitle

\section{Introduction}

This paper is concerned with the prescribed-mass construction of standing wave solutions
for the focusing nonlinear Schr\"odinger equation (NLS)
\begin{equation}\label{eq:standing-wave-intro}
        -\Delta_{x,y}u+\omega u=u^{1+4/d},\quad \omega>0
\end{equation}
on the waveguide manifold
\[
        \R^d_x\times\T_y,
        \qquad
        \T=\R/2\pi\mathbb Z.
\]
The nonlinearity has order \(4/d\) and hence the equation is mass-critical with respect to the Euclidean variable \(x\in\R^d\), while the underlying
space contains an additional periodic direction.  By simply neglecting the periodic direction, \eqref{eq:standing-wave-intro} reduces to the standard focusing mass-critical NLS on the Euclidean space \(\R^d\), whose solutions are unique up to the standard symmetries, and their mass is fixed.  This follows from the sharp Gagliardo--Nirenberg theory of Weinstein and the uniqueness theorem of Kwong; see \cite{weinstein,Kwong_uniqueness}.  

It it hence an interesting problem whether solutions of \eqref{eq:standing-wave-intro} that are not necessarily constant along the periodic direction exist. Starting from the Euclidean decaying branch, Dancer \cite{DancerSolution} used local bifurcation methods to construct positive solutions for a wide class of elliptic problems--covering also \eqref{eq:standing-wave-intro}--which decay in the noncompact variables and are periodic, but not constant, in one additional variable.  These solutions provide a fundamental
example of how adding a periodic direction can destroy the rigidity of the Euclidean problem.  However, the bifurcation construction does not directly provide any energetic characterization of the resulting solutions. In particular, it does not show whether they minimize a natural energy functional under a prescribed mass constraint, nor does it describe how their masses are related to the Euclidean mass-critical ground state on $\R^d$.

Our goal is to provide precisely such a normalized variational description for solutions of \eqref{eq:standing-wave-intro}. The point of view is different from Dancer's bifurcation argument: instead of following a local branch of solutions, we characterize the solutions as ground states of a constrained prescribed-mass problem.  More precisely,
following the semivirial-vanishing geometry developed in our previous works \cite{luo2021sharp,Luo_inter,Luo_energy_crit,Luo_LegendreFenchel}, we study the
variational problem
\[
        m_c
        :=
        \inf\{E(u):M(u)=c,\ K(u)=0\},
\]
where the mass \(M\), the energy \(E\), and the semivirial functional \(K\) are
defined precisely in Section~\ref{sec:preliminaries}.  In the Euclidean
mass-critical problem, normalized ground states occur only at the single mass
\(\widehat M(Q)\).  The main result of this paper shows that the waveguide
geometry breaks this rigidity in a strong way: for every mass below
\(|\T|\widehat M(Q)=2\pi\widehat M(Q)\), the level \(m_c\) is attained, and the
corresponding optimizer gives a positive-frequency solution of
\eqref{eq:standing-wave-intro} after a mass-preserving scaling.

NLS-models on waveguide manifolds arise naturally in nonlinear optics
and related physical models; see for instance
\cite{waveguide_ref_1,waveguide_ref_2,waveguide_ref_3}.  They have also
attracted substantial attention in dispersive PDE because the Euclidean
directions provide dispersion while the compact directions retain periodic
dynamics.  Foundational works on scattering and global dynamics include
\cite{HaniPausader,TNCommPDE,TzvetkovVisciglia2016}, while the variational
theory of ground states on product spaces was developed, among others, in
\cite{TTVproduct2014}.  In the focusing case, semivirial-vanishing geometry was
introduced in our previous works
\cite{luo2021sharp,Luo_inter,Luo_energy_crit,Luo_LegendreFenchel}.  This
framework uses the virial structure only in the dispersive Euclidean variables
and leads to variational thresholds adapted to the mixed geometry.  We also refer to
\cite{RmT1,R1T1Scattering,ZhaoZheng2021,LuoSIMA,ForcellaLuoZhao,LuoCriticalScattering} for further recent development in this direction.

The most relevant predecessor of the present work is
\cite{Luo_LegendreFenchel}.  There we adapted semivirial-vanishing geometry to
the mass-critical endpoint by proving a Legendre--Fenchel identity, inspired by
the duality mechanism of \cite{actionvsenergy}.  This allowed us to obtain
normalized ground states for a sequence of masses tending to zero.  However,
the full interval of masses remained open.  The obstruction identified in
\cite{Luo_LegendreFenchel} is specific to the mass-critical setting: the standard
mass-preserving scaling leaves the useful energy on the semivirial constraint
invariant, and the usual mass-supercritical natural-constraint argument no
longer produces the required monotonicity or compactness.  In particular,
\cite{Luo_LegendreFenchel} proved only that \(c\mapsto m_c\) is monotone
decreasing and lower semicontinuous.

The first new ingredient of this paper is a strict monotonicity principle for
the prescribed-mass level \(m_c\), which is the main content of
Section~\ref{sec:projected-variations}.  The non-strict monotonicity in
\cite{Luo_LegendreFenchel} is obtained by a surgery argument which adds extra
mass far away and therefore loses all nonlinear interaction.  Such a
construction cannot yield a strict inequality.  We instead perturb an attained
minimizer in an overlapping way and then project the perturbation back onto the
semivirial constraint.  The projected first variations define two linear
functionals: one measures the first-order change of mass, and the other measures
the first-order change of the \(y\)-energy.  If a direction increases the mass
and decreases the projected \(y\)-energy, then a local strict decrease of \(m_c\)
follows.  If no such direction exists, the optimizer must satisfy a
zero-frequency semilinear elliptic equation after an anisotropic rescaling.  This
alternative is ruled out by a Liouville theorem after periodic extension to
\(\R^{d+1}\).  In a nutshell, the key point is the strict monotonicity
\[
        0<b<a<2\pi\widehat M(Q)
        \quad\Longrightarrow\quad
        m_a<m_b.
\]
This strict inequality eliminates the flat intervals in which mass could escape
at large \(x\)-scale with zero \(y\)-energy, and it restores compactness for all
prescribed masses below the threshold.

The second difficulty is to prove that an optimizer of \(m_c\) is actually a
standing wave solution.  In mass-supercritical problems this is usually a
consequence of the natural-constraint structure detected by scaling.  Here that
argument degenerates because the mass-critical scaling preserves both the mass
constraint and the value of the energy on \(K=0\).  Section~\ref{sec:standing-wave}
overcomes this by returning to the projected first variations.  At a true
minimizer, tangent perturbations preserving the projected mass cannot lower the
projected \(y\)-energy.  A codimension-one multiplier argument then yields a
proportionality relation between the two projected linear functionals.  This
relation is precisely the weak Euler--Lagrange equation.  The multiplier has the
correct sign because of the strict projected direction obtained earlier, and the
zero-frequency alternative is again excluded.  After a mass-preserving
\(x\)-scaling, the optimizer solves \eqref{eq:standing-wave-intro} with positive
frequency.

Let us state the main theorem in the notation of Section~\ref{sec:preliminaries}.
Here \(Q\) is the Euclidean mass-critical ground state defined by
\eqref{eq:Q-eqn}, and \(\widehat M(Q)\) denotes its mass on \(\R^d\).

\begin{theorem}\label{thm:main}
Let $d\ge1$ and let
\(
        0<c<2\pi\whM(Q).
\)
Then the variational problem $m_c$ has a nonnegative optimizer.  Moreover,
after a mass-preserving scaling in the $x$-variable, every nonnegative optimizer
solves the standing wave equation \eqref{eq:standing-wave-intro} for some $\omega>0$.
\end{theorem}

\begin{remark}
By strong maximum principle, a nonnegative solution given by Theorem \ref{thm:main} is in fact positive everywhere on $\R^d\times\T$.
\end{remark}

Thus the mass rigidity of the Euclidean mass-critical ground state is replaced,
on the waveguide, by an all-mass existence theory below the Euclidean threshold
multiplied by the period length.  In particular, this answers the open question
left in our previous work \cite{Luo_LegendreFenchel}.  We also mention the
recent work \cite{luo2026dancer1}, where a related energy method is developed
for energy-critical Dancer-type solutions.  Both that work and the present one
rely on strict inequalities to recover compactness, but the mechanisms are
different: \cite{luo2026dancer1} proves a strict sub-bubbling bound below the
Euclidean Sobolev threshold, whereas the present paper proves strict
monotonicity of the mass-critical semivirial level $m_c$.

The paper is organized as follows.  Section~\ref{sec:preliminaries} collects the
variational definitions and the inputs from our previous work
\cite{Luo_LegendreFenchel}.  Section~\ref{sec:projected-variations} develops the
projected first-variation calculus and proves the local strict-decrease
mechanism.  Section~\ref{sec:strict-existence} proves strict monotonicity of
\(m_c\) and all-mass existence of normalized ground states.
Section~\ref{sec:standing-wave} identifies optimizers as positive-frequency
standing waves.  Section~\ref{sec:conclusion} concludes the paper by assembling
the previous results into the proof of Theorem~\ref{thm:main}.

\section{Preliminaries}\label{sec:preliminaries}

\subsection{Notation and definitions}
We use the notation $A\lesssim B$ if there exists a constant $C>0$ such that $A\le CB$.  Similarly, $A\gtrsim B$ means $B\lesssim A$, and $A\sim B$ means both $A\lesssim B$ and $B\lesssim A$. 

Throughout the paper we work in the mass-critical case
\[
        \alpha=\frac4d,
        \qquad
        2_*:=\alpha+2=2+\frac4d.
\]
For $u\in H^1(\R^d\times\T)$ define
\begin{align*}
        M(u)&:=\norm[L^2_{x,y}]{u}^2,\\
        E(u)&:=\frac12\norm[L^2_{x,y}]{\nabla_{x,y}u}^2
        -\frac1{2_*}\norm[L^{2_*}_{x,y}]{u}^{2_*},\\
        K(u)&:=\norm[L^2_{x,y}]{\nabla_xu}^2
        -\frac d{d+2}\norm[L^{2_*}_{x,y}]{u}^{2_*}.
\end{align*}
We also write
\[
        A(u):=\norm[L^2_{x,y}]{\nabla_xu}^2,
        \qquad
        B(u):=\norm[L^2_{x,y}]{\partial_yu}^2,
        \qquad
        P(u):=\norm[L^{2_*}_{x,y}]{u}^{2_*}.
\]
For $c>0$, define the sets
\[
        S(c):=\{u\in H^1(\R^d\times\T):M(u)=c\},
        \qquad
        V(c):=\{u\in S(c):K(u)=0\},
\]
and the variational problem $m_c$ by
\[
        m_c:=\inf\{E(u):u\in V(c)\}.
\]
For functions on $\R^d$, we use hatted notation:
\begin{align*}
        \whM(v):=\norm[L^2(\R^d)]{v}^2,\qquad
        \whK(v):=\norm[L^2(\R^d)]{\nabla_xv}^2
        -\frac d{d+2}\norm[L^{2_*}(\R^d)]{v}^{2_*}.
\end{align*}


The following identity is used repeatedly.  If $u\in V(c)$, then
\begin{equation}\label{eq:energy-reduction}
        E(u)=\frac12B(u).
\end{equation}
Indeed, $K(u)=0$ gives
\(
        A(u)=\frac d{d+2}P(u),
\)
and since $1/2_*=d/[2(d+2)]$, we have
\(
        \frac12 A(u)-\frac1{2_*}P(u)=0
\), implying \eqref{eq:energy-reduction}.

The mass-critical $x$-scaling is defined by
\begin{equation}\label{eq:x-scaling}
        u_t(x,y):=t^{d/2}u(tx,y),
        \qquad t>0.
\end{equation}
It satisfies
\begin{equation}\label{eq:x-scaling-identities}
        M(u_t)=M(u),\qquad
        B(u_t)=B(u),\qquad
        K(u_t)=t^2K(u).
\end{equation}
In particular, $u\in V(c)$ implies $u_t\in V(c)$ and $E(u_t)=E(u)$ for any $t\in(0,\infty)$.

We shall also use the sharp Euclidean mass-critical theory. Denote by $Q$ the unique (up to symmetries) positive radial solution of the focusing mass-critical NLS
\begin{equation}\label{eq:Q-eqn}
        -\Delta Q+Q=Q^{2_*-1}
        \qquad\text{on }\R^d.
\end{equation}
The sharp Gagliardo--Nirenberg inequality of Weinstein \cite{weinstein} states that
\begin{equation}\label{eq:sharp-GN}
        \norm[L^{2_*}(\R^d)]{v}^{2_*}
        \le
        \frac{d+2}{d}
        \left(\frac{\whM(v)}{\whM(Q)}\right)^{2/d}
        \norm[L^2(\R^d)]{\nabla v}^2.
\end{equation}
Moreover, by the classical uniqueness theorem of Kwong \cite{Kwong_uniqueness}, every nontrivial nonnegative solution of
\[
        -\Delta v+\omega v=v^{2_*-1},
        \qquad \omega>0,
\]
has mass $\whM(Q)$ after the natural mass-critical scaling.

\subsection{Some useful auxiliary tools}\label{new2}

We record and prove several elementary consequences that will be used repeatedly.

\begin{lemma}[Nonemptiness, monotonicity and lower semicontinuity, \cite{Luo_LegendreFenchel}]\label{prop:jga-known}
Let $0<c<2\pi\whM(Q)$.  Then $V(c)\neq\emptyset$.  Moreover, the map
\(
        c\mapsto m_c
\)
is monotone decreasing and lower semicontinuous on $(0,2\pi\whM(Q))$.
\end{lemma}

\begin{lemma}[Pseudo-compactness along a minimizing sequence]\label{lem:left-attained}
For every $c\in(0,2\pi\whM(Q))$, there exists some
\(
        \bar c\in(0,c]
\)
and an optimizer $u_{\bar c}\in V(\bar c)$ of $m_{\bar c}$ such that
\[
        m_{\bar c}=m_c.
\]
\end{lemma}

\begin{proof}
Choose $c_n\downarrow c$ and $u_n\in V(c_n)$ such that
\[
        E(u_n)\le m_{c_n}+o_n(1).
\]
Using the scaling \eqref{eq:x-scaling}, we may assume
\(
        A(u_n)=1.
\)
Since $K(u_n)=0$, this gives
\begin{align}\label{new}
        P(u_n)=\frac{d+2}{d}.
\end{align}
By \eqref{eq:energy-reduction},
\[
        B(u_n)=2E(u_n)\le 2m_{c_n}+o_n(1).
\]
The monotonicity of $c\mapsto m_c$ from Lemma \ref{prop:jga-known} implies that $(m_{c_n})_n$ is bounded above.  Hence $(u_n)_n$ is bounded in $H^1(\R^d\times\T)$ and \eqref{new} implies
\[
        \liminf_{n\to\infty}\norm[L^{2_*}_{x,y}]{u_n}>0.
\]
By standard compact arguments (see e.g. the proof of \cite[Thm. 1.1]{TTVproduct2014} in the context of waveguide manifolds), this yields translations $x_n\in\R^d$ and a nonzero function $u$ such that, after passing to a subsequence,
\[
        u_n(\cdot+x_n,\cdot)\rightharpoonup u
        \qquad\text{weakly in }H^1(\R^d\times\T).
\]
Relabelling the translated sequence, set $\bar c:=M(u)$.  Then
\[
        0<\bar c\le \liminf_{n\to\infty}M(u_n)=c.
\]
Using the standard arguments for deducing an optimizer of the minimization problem (see e.g. the proof of \cite[Thm. 1.1]{Luo_inter}) we infer that 
\[
        u\in V(\bar c),
        \qquad
        E(u)=m_{\bar c}.
\]
Furthermore, weak lower semicontinuity of $B$ and \eqref{eq:energy-reduction} give
\[
        m_{\bar c}
        =E(u)
        =\frac12B(u)
        \le \frac12\liminf_{n\to\infty}B(u_n)
        =\liminf_{n\to\infty}E(u_n)
        =\lim_{n\to\infty}m_{c_n}.
\]
Additionally, the lower semicontinuity and decreasing monotonicity of $c\mapsto m_c$ from Lemma \ref{prop:jga-known} gives
\[
        \lim_{n\to\infty}m_{c_n}=m_c.
\]
Thus $m_{\bar c}\le m_c$.  Since $\bar c\le c$, Lemma \ref{prop:jga-known} also implies $m_c\le m_{\bar c}$ and consequently $m_{\bar c}=m_c$.
\end{proof}

Finally, we shall also invoke the Liouville theorem of Gidas--Spruck \cite{Liouville_type}: if $D\ge3$ and
\[
        1<q<\frac{D+2}{D-2},
\]
then there is no positive classical solution of
\[
        -\Delta U=U^q
        \qquad\text{on }\R^D.
\]
In dimension $D=2$, the corresponding finite-exponent Liouville theorem is standard; see, for example, \cite{Liouville2}.

\section{Projected first variations}\label{sec:projected-variations}

In this section we work at an attained mass.  The purpose is to understand whether one can increase the mass and decrease the energy after projecting back to the semivirial constraint $K=0$.

\subsection{Projection to the constraint}
We start by showing that an optimizer can always be assumed nonnegative.
\begin{lemma}\label{lem:nonnegative-optimizer}
Suppose $m_b$ is attained for some $b\in(0,2\pi\whM(Q))$.  Then $m_b$ has a nonnegative optimizer.  Moreover, every nonnegative optimizer $u\in V(b)$ satisfies
\(
        B(u)>0.
\)
\end{lemma}

\begin{proof}
Let $u\in V(b)$ be an optimizer.  First assume that $B(u)=0$.  Then $u$ is independent of $y$, say $u(x,y)=v(x)$.  Since $b<2\pi\whM(Q)$, we have
\(
        \whM(v)<\whM(Q).
\)
By the sharp Gagliardo-Nirenberg inequality \eqref{eq:sharp-GN}, $\whK(v)>0$ unless $v\equiv0$.  Thus
\[
        K(u)=2\pi\whK(v)>0,
\]
which contradicts $u\in V(b)$.  Hence $B(u)>0$ for every optimizer.

Set $w:=|u|$.  Then
\[
        M(w)=M(u)=b,
        \qquad
        P(w)=P(u),
\]
and by the diamagnetic inequality,
\[
        A(w)\le A(u),
        \qquad
        B(w)\le B(u).
\]
Therefore
\(
        K(w)\le K(u)=0.
\)
If $K(w)<0$, then there exists $\tau\in(0,1)$ such that $K(\tau w)=0$.  Since $K(\tau w)=0$, \eqref{eq:energy-reduction} gives
\[
        E(\tau w)=\frac12B(\tau w)=\frac{\tau^2}{2}B(w)
        \le \tau^2 E(u)<E(u)=m_b.
\]
But $M(\tau w)=\tau^2b<b$, and monotonicity of $c\mapsto m_c$ gives
\[
        m_{\tau^2b}\ge m_b.
\]
This contradicts $m_{\tau^2b}\le E(\tau w)<m_b$.  Hence $K(w)=0$.  Consequently $w\in V(b)$ and
\[
        E(w)=\frac12B(w)\le\frac12B(u)=m_b.
\]
Thus $w$ is a nonnegative optimizer.
\end{proof}
By Lemma \ref{lem:nonnegative-optimizer}, we may from now on assume that all optimizers are nonnegative. Let $u\in V(b)$ be a nonnegative optimizer of $m_b$.  Set
\[
        A:=A(u),
        \qquad
        B:=B(u).
\]
Then $A,B>0$ by Lemma~\ref{lem:nonnegative-optimizer}.

For a real-valued test function $h\in C_c^\infty(\R^d\times\T)$, define
\[
        v_\varepsilon:=u+\varepsilon h.
\]
For $|\varepsilon|$ sufficiently small, there is a unique $\tau_\varepsilon>0$ close to $1$ such that
\[
        K(\tau_\varepsilon v_\varepsilon)=0.
\]
Indeed, the function $(\tau,\varepsilon)\mapsto K(\tau v_\varepsilon)$ is $C^1$ near $(1,0)$ and
\[
        \partial_\tau K(\tau u)\big|_{\tau=1}
        =(2-2_*)A=-\frac4d A\ne0.
\]
Thus the implicit function theorem applies.

The first variation of $K$ at $u$ is
\begin{equation}\label{eq:K-prime}
        K'(u)[h]
        =2\int \nabla_xu\cdot\nabla_xh\dd x\dd y
        -2\int u^{2_*-1}h\dd x\dd y.
\end{equation}
We now compute the first-order expansion of \(\tau_\varepsilon\).  Set
\[
        F(\tau,\varepsilon):=K(\tau v_\varepsilon)
        =
        K\bigl(\tau(u+\varepsilon h)\bigr).
\]
Then \(F(1,0)=K(u)=0\).  Moreover, using
\[
        K(\tau u)
        =
        \tau^2A(u)
        -
        \frac d{d+2}\tau^{2_*}P(u),
\]
we obtain
\[
        \partial_\tau F(1,0)
        =
        2A(u)
        -
        2_*\,\frac d{d+2}P(u).
\]
Since \(K(u)=0\), we have
\(
        A(u)=\frac d{d+2}P(u)
\)
and consequently
\[
        \partial_\tau F(1,0)
        =
        (2-2_*)A(u)
        =
        -\frac4d A(u)\neq0.
\]
On the other hand,
\[
        \partial_\varepsilon F(1,0)
        =
        K'(u)[h].
\]
The identity \(F(\tau_\varepsilon,\varepsilon)=0\) holds for all sufficiently small
\(\varepsilon\).  Differentiating this identity at \(\varepsilon=0\), and using
\(\tau_0=1\), gives
\[
        0
        =
        \partial_\tau F(1,0)\,\tau_\varepsilon'\big|_{\varepsilon=0}
        +
        \partial_\varepsilon F(1,0)
        =
        -\frac4d A\,\tau_\varepsilon'\big|_{\varepsilon=0}
        +
        K'(u)[h].
\]
Hence
\[
        \tau_\varepsilon'\big|_{\varepsilon=0}
        =
        \frac{d}{4A}K'(u)[h].
\]
Since \(F\) is \(C^2\) in a neighbourhood of \((1,0)\), the implicit function
\(\varepsilon\mapsto\tau_\varepsilon\) is \(C^2\), and Taylor expansion yields
\begin{equation}\label{eq:tau-expansion-general}
        \tau_\varepsilon
        =
        1+\frac{d}{4A}K'(u)[h]\varepsilon+O(\varepsilon^2).
\end{equation}

We next derive the expansions of the mass and the \(y\)-energy after projection.
First,
\[
        M(v_\varepsilon)
        =
        \int |u+\varepsilon h|^2
        =
        b+2\varepsilon\int uh\,dxdy+O(\varepsilon^2),
\]
and \eqref{eq:tau-expansion-general} implies
\[
        \tau_\varepsilon^2
        =
        1+\frac{d}{2A}K'(u)[h]\varepsilon+O(\varepsilon^2).
\]
Therefore
\begin{align*}
        M(\tau_\varepsilon v_\varepsilon)
        &=
        \tau_\varepsilon^2M(v_\varepsilon) \\
        &=
        \left(
        1+\frac{d}{2A}K'(u)[h]\varepsilon+O(\varepsilon^2)
        \right)
        \left(
        b+2\varepsilon\int uh\,dxdy+O(\varepsilon^2)
        \right) \\
        &=
        b
        +
        2\left(
        \int uh\,dxdy+\frac{db}{4A}K'(u)[h]
        \right)\varepsilon
        +
        O(\varepsilon^2).
\end{align*}
Similarly,
\[
        B(v_\varepsilon)
        =
        \int |\partial_yu+\varepsilon\partial_yh|^2
        =
        B+2\varepsilon\int \partial_yu\,\partial_yh\,dxdy+O(\varepsilon^2),
\]
and hence
\begin{align*}
        B(\tau_\varepsilon v_\varepsilon)
        &=
        \tau_\varepsilon^2B(v_\varepsilon) \\
        &=
        B
        +
        2\left(
        \int \partial_yu\,\partial_yh\,dxdy+\frac{dB}{4A}K'(u)[h]
        \right)\varepsilon
        +
        O(\varepsilon^2).
\end{align*}
Thus
\begin{align}
        M(\tau_\varepsilon v_\varepsilon)
        &=b+2\mathcal M_u(h)\varepsilon+O(\varepsilon^2),
        \label{eq:projected-M}\\
        B(\tau_\varepsilon v_\varepsilon)
        &=B+2\mathcal B_u(h)\varepsilon+O(\varepsilon^2),
        \label{eq:projected-B}
\end{align}
where the projected first variations are
\begin{align*}
        \mathcal M_u(h)
        &:=\int uh\,dxdy+\frac{db}{4A}K'(u)[h],\\
        \mathcal B_u(h)
        &:=\int \partial_yu\,\partial_yh\,dxdy+\frac{dB}{4A}K'(u)[h].
\end{align*}

We refer a test function $h$ to as a {\it{strict projected direction}} if the energies $ \mathcal M_u(h)$ and $ \mathcal B_u(h)$ have opposite signs. That the existence of a strict projected direction leads to local strict monotonicity is guaranteed by the following lemma. 

\begin{lemma}[Strict projected directions]\label{lem:strict-direction}
If there exists $h\in C_c^\infty(\R^d\times\T)$ such that
\[
        \mathcal M_u(h)>0,
        \qquad
        \mathcal B_u(h)<0,
\]
then there exists $\delta>0$ such that
\[
        m_a<m_b
        \qquad\text{for all }a\in(b,b+\delta).
\]
\end{lemma}

\begin{proof}
Let $w_\varepsilon:=\tau_\varepsilon v_\varepsilon$, then $K(w_\varepsilon)=0$.  By \eqref{eq:projected-M},
\(
        M(w_\varepsilon)>b
\)
for all sufficiently small $\varepsilon>0$, and the map $\varepsilon\mapsto M(w_\varepsilon)$ has positive derivative at $0$.  Combining the intermediate value theorem, this implies that the image of $\varepsilon\mapsto M(w_\varepsilon)$ contains the interval $(b,b+\delta)$ with some $\delta>0$.

By \eqref{eq:projected-B},
\(
        B(w_\varepsilon)<B(u)
\)
for all sufficiently small $\varepsilon>0$.  Since $K(w_\varepsilon)=0$, \eqref{eq:energy-reduction} gives
\[
        E(w_\varepsilon)=\frac12B(w_\varepsilon)<\frac12B(u)=E(u)=m_b.
\]
Thus, for every $a\in(b,b+\delta)$, choosing $\varepsilon(a)$ with $M(w_{\varepsilon(a)})=a$ yields
\[
        m_a\le E(w_{\varepsilon(a)})<m_b.
\]
\end{proof}

\subsection{Failure of strict directions}

The next lemmas show that strict projected directions must exist.  The argument is a first-order alternative: if no such direction exists, the optimizer solves a zero-frequency equation, which is impossible by Liouville's theorem.

\begin{lemma}\label{lem:calM-nonzero}
Let $u\in V(b)$ be a nonnegative optimizer of $m_b$, with $0<b<2\pi\whM(Q)$.  Then
\[
        \mathcal M_u\not\equiv0
        \qquad\text{on }C_c^\infty(\R^d\times\T).
\]
\end{lemma}

\begin{proof}
Assume, to the contrary, that $\mathcal M_u\equiv0$.  Then, for every test function $h$,
\[
        \int uh\dd x\dd y+\frac{db}{4A}K'(u)[h]=0.
\]
Using \eqref{eq:K-prime}, we obtain
\[
        \int uh\dd x\dd y+\frac{db}{2A}\int\nabla_xu\cdot\nabla_xh\dd x\dd y
        -
        \frac{db}{2A}\int u^{2_*-1}h\dd x\dd y=0.
\]
Set
\[
        \eta:=\frac{db}{2A}>0.
\]
Then $u$ satisfies, in the sense of distributions,
\begin{equation}\label{eq:slice-equation}
        -\eta\Delta_xu+u=\eta u^{2_*-1}.
\end{equation}
There are no $y$-derivatives in \eqref{eq:slice-equation}.  Hence, for a.e. $y\in\T$, the slice $v_y(x):=u(x,y)$ is a nonnegative $H^1(\R^d)$ solution of
\[
        -\eta\Delta_xv_y+v_y=\eta v_y^{2_*-1}.
\]
Equivalently,
\[
        -\Delta_xv_y+\eta^{-1}v_y=v_y^{2_*-1}.
\]
By the Euclidean mass-critical elliptic theory, every nontrivial nonnegative solution of this equation has $L^2(\R^d)$-mass equal to $\whM(Q)$.  Therefore the slice mass
\[
        y\mapsto \norm[L^2(\R^d)]{u(\cdot,y)}^2
\]
takes values only in $\{0,\whM(Q)\}$ for a.e. $y$.

Since $u\in H^1(\T;L^2(\R^d))$, the map $y\mapsto u(\cdot,y)$ is continuous from $\T$ to $L^2(\R^d)$.  Thus the slice mass is continuous.  Because $\T$ is connected, it is constant.  If it is identically zero, then $u\equiv0$, contradicting $M(u)=b>0$.  If it is identically $\whM(Q)$, then
\[
        M(u)=2\pi\whM(Q),
\]
contradicting $b<2\pi\whM(Q)$.  Hence $\mathcal M_u$ cannot vanish identically.
\end{proof}

\begin{lemma}[First-order alternative]\label{lem:first-order-alternative}
Let $u\in V(b)$ be a nonnegative optimizer of $m_b$.  Then one of the following alternatives holds:
\begin{enumerate}[label=\emph{(\roman*)}]
\item there exists $h\in C_c^\infty(\R^d\times\T)$ such that
\[
        \mathcal M_u(h)>0,
        \qquad
        \mathcal B_u(h)<0;
\]
\item
\[
        \mathcal B_u(h)=0
        \qquad
        \text{for every }h\in C_c^\infty(\R^d\times\T).
\]
\end{enumerate}
\end{lemma}

\begin{proof}
Assume that alternative (i) fails.  Then
\begin{equation}\label{eq:no-plus-minus}
        \mathcal M_u(h)>0
        \quad\Longrightarrow\quad
        \mathcal B_u(h)\ge0.
\end{equation}
We first show that
\begin{equation}\label{eq:no-minus-minus}
        \mathcal M_u(h)<0
        \quad\Longrightarrow\quad
        \mathcal B_u(h)\ge0.
\end{equation}
Indeed, if $\mathcal M_u(h)<0$ and $\mathcal B_u(h)<0$, then for small $\varepsilon>0$ the projected function $w_\varepsilon=\tau_\varepsilon(u+\varepsilon h)$ satisfies
\[
        M(w_\varepsilon)<b,
        \qquad
        E(w_\varepsilon)<m_b.
\]
Thus
\[
        m_{M(w_\varepsilon)}<m_b.
\]
But $M(w_\varepsilon)<b$, and the monotonicity from Lemma~\ref{prop:jga-known} gives
\[
        m_{M(w_\varepsilon)}\ge m_b,
\]
a contradiction.  Hence \eqref{eq:no-minus-minus} holds.

Now let $h$ be such that $\mathcal M_u(h)>0$.  By \eqref{eq:no-plus-minus}, $\mathcal B_u(h)\ge0$.  Applying \eqref{eq:no-minus-minus} to $-h$, we get $\mathcal B_u(-h)\ge0$, hence $\mathcal B_u(h)\le0$.  Therefore $\mathcal B_u(h)=0$.  The same argument applies to every $h$ with $\mathcal M_u(h)<0$.

It remains to consider $h$ with $\mathcal M_u(h)=0$.  By Lemma~\ref{lem:calM-nonzero}, choose $g$ with $\mathcal M_u(g)>0$ after replacing $g$ by $-g$ if necessary.  For every sufficiently small $t>0$,
\[
        \mathcal M_u(h+tg)>0.
\]
Hence $\mathcal B_u(h+tg)=0$.  Since $\mathcal B_u(g)=0$, it follows that $\mathcal B_u(h)=0$.  Therefore alternative (ii) holds.
\end{proof}

\begin{lemma}[Zero-frequency alternative]\label{lem:zero-frequency}
Let $u\in V(b)$ be a nonnegative optimizer of $m_b$.  If alternative \emph{(ii)} in Lemma~\ref{lem:first-order-alternative} holds, then $u$ satisfies
\begin{equation}\label{eq:anisotropic-zero}
        -\theta\Delta_xu-\partial_y^2u=\theta u^{2_*-1}
\end{equation}
in the sense of distributions, where
\[
        \theta:=\frac{dB}{2A}>0.
\]
\end{lemma}

\begin{proof}
Alternative {(ii)} means
\[
        \int \partial_yu\,\partial_yh\dd x\dd y+\frac{dB}{4A}K'(u)[h]=0
\]
for every test function $h$.  Using \eqref{eq:K-prime} and setting $\theta=dB/(2A)$, we obtain
\[
        \int \partial_yu\,\partial_yh\dd x\dd y
        +\theta\int \nabla_xu\cdot\nabla_xh\dd x\dd y
        -\theta\int u^{2_*-1}h\dd x\dd y=0.
\]
This is exactly the weak formulation of \eqref{eq:anisotropic-zero}.  The positivity of $\theta$ follows from $A,B>0$.
\end{proof}

\begin{lemma}[Liouville exclusion]\label{lem:liouville}
There is no nontrivial nonnegative $H^1(\R^d\times\T)$ solution of
\[
        -\theta\Delta_xu-\partial_y^2u=\theta u^{2_*-1}
\]
with $\theta>0$.
\end{lemma}

\begin{proof}
Let $t:=\sqrt\theta$ and define
\[
        v(x,y):=t^{d/2}u(tx,y).
\]
Using $\frac d2(2_*-1)=\frac d2+2$, a direct computation gives
\[
        -\Delta_{x,y}v=v^{2_*-1}
        \qquad\text{on }\R^d\times\T.
\]
By elliptic regularity and the strong maximum principle, $v$ is either identically zero or strictly positive.  Since $u$ is nontrivial, $v>0$.

Extend $v$ periodically in the $y$-variable to a positive classical solution on $\R^{d+1}$.  Set
\[
        D:=d+1,
        \qquad
        q:=2_*-1=1+\frac4d.
\]
If $d=1$, then $D=2$ and $q=5$ is a finite exponent; the two-dimensional Liouville theorem quoted in Section~\ref{new2} excludes positive entire solutions of $-\Delta v=v^q$.  If $d\ge2$, then $D\ge3$ and
\[
        1<q=1+\frac4d<1+\frac4{d-1}=\frac{D+2}{D-2}.
\]
Thus $q$ is Sobolev-subcritical in dimension $D$, and the Gidas--Spruck Liouville theorem quoted in Section~\ref{new2} again excludes such a positive entire solution.  This contradiction proves the lemma.
\end{proof}

Summarizing Lemma \ref{lem:nonnegative-optimizer} to Lemma \ref{lem:liouville}, we deduce immediately the existence of a strict projected direction, and the local strict decrease of $c\mapsto m_c$ at attained masses.

\begin{lemma}[Existence of a strict projected direction]
\label{lem:strict-projected-direction-exists}
Let \(u\in V(b)\) be a nonnegative optimizer of \(m_b\), where
\(
        0<b<2\pi\widehat M(Q).
\)
Then there exists \(h_0\in C_c^\infty(\mathbb R^d\times\mathbb T)\) such that
\[
        \mathcal M_u(h_0)>0,
        \qquad
        \mathcal B_u(h_0)<0.
\]
\end{lemma}

\begin{lemma}[Local strict decrease at attained masses]\label{prop:local-strict}
Let $b\in(0,2\pi\whM(Q))$, and suppose that $m_b$ is attained.  Then there exists $\delta>0$ such that
\[
        m_a<m_b
        \qquad\text{for every }a\in(b,b+\delta).
\]
\end{lemma}

\section{Strict monotonicity and existence of optimizers at any mass}\label{sec:strict-existence}

The local result obtained in the previous section says that the level \(m_c\) strictly decreases to the right of every mass at which the infimum is attained.  This section globalizes that local information.  The compactness input from Lemma~\ref{lem:left-attained} ensures that if a flat interval of \(m_c\) existed, then the left endpoint of such a flat interval would be attained.  The local strict-decrease lemma would then immediately contradict flatness.  This gives strict monotonicity on the whole interval below the Euclidean mass threshold.

\begin{theorem}[Strict monotonicity]\label{thm:strict-monotonicity}
For every
\(
        0<b<a<2\pi\whM(Q),
\)
one has
\(
        m_a<m_b.
\)
\end{theorem}

\begin{proof}
Suppose, to the contrary, that strict monotonicity fails.  Since $m_c$ is monotone decreasing, there exist
\(
        0<b<a<2\pi\whM(Q)
\)
such that
\(
        m_a=m_b.
\)
Then
\[
        m_\rho=m_b
        \qquad\text{for every }\rho\in[b,a].
\]
By Lemma~\ref{lem:left-attained}, there exist $\bar b\in(0,b]$ and an optimizer $u_{\bar b}\in V(\bar b)$ of $m_{\bar b}$ such that
\[
        m_{\bar b}=m_b.
\]
By monotonicity, $m_\rho$ is constant on $[\bar b,a]$:
\[
        m_\rho=m_{\bar b}
        \qquad\text{for every }\rho\in[\bar b,a].
\]
However, Lemma~\ref{prop:local-strict} applied at the attained mass $\bar b$ gives $\delta>0$ such that
\[
        m_\rho<m_{\bar b}
        \qquad\text{for }\rho\in(\bar b,\bar b+\delta).
\]
Choosing $\rho\in(\bar b,\min\{a,\bar b+\delta\})$ gives a contradiction.  Hence strict monotonicity holds.
\end{proof}

We next use strict monotonicity to upgrade the pseudo-compactness statement of
Lemma~\ref{lem:left-attained} to genuine compactness at the prescribed mass. Indeed, Lemma~\ref{lem:left-attained} already gives an optimizer at some mass \(\bar c\le c\) with the same energy level.  Strict monotonicity rules out the possibility \(\bar c<c\), and hence prevents loss of mass.

\begin{theorem}[Existence of normalized ground states]\label{thm:existence}
For every
\[
        0<c<2\pi\whM(Q),
\]
the variational problem $m_c$ has an optimizer.  Moreover, $m_c$ has a nonnegative optimizer.
\end{theorem}

\begin{proof}
By Lemma~\ref{lem:left-attained}, there exist $\bar c\in(0,c]$ and an optimizer $u_{\bar c}\in V(\bar c)$ of $m_{\bar c}$ such that
\[
        m_{\bar c}=m_c.
\]
If $\bar c<c$, then Theorem~\ref{thm:strict-monotonicity} gives
\[
        m_c<m_{\bar c},
\]
which contradicts $m_{\bar c}=m_c$.  Hence $\bar c=c$, and $u_{\bar c}$ is an optimizer of $m_c$.  Finally, Lemma~\ref{lem:nonnegative-optimizer} gives a nonnegative optimizer.
\end{proof}

\section{Optimizers characterized as standing wave solutions}\label{sec:standing-wave}

It remains to show that an optimizer is a standing wave solution for the focusing NLS on waveguide manifold, after applying the mass-preserving scaling in the $x$-variable.  The point is that the first-order variations above also encode the Euler--Lagrange equation.

\begin{lemma}[Projected tangent optimality]\label{lem:tangent-optimality}
Let $u\in V(c)$ be a nonnegative optimizer of $m_c$.  If
\[
        \mathcal M_u(h)=0,
\]
then
\[
        \mathcal B_u(h)=0.
\]
Consequently, there exists a real number $\ell<0$ such that
\begin{equation}\label{eq:calB-ell-calM}
        \mathcal B_u(h)=\ell\,\mathcal M_u(h)
        \qquad
        \text{for every }h\in C_c^\infty(\R^d\times\T).
\end{equation}
\end{lemma}

\begin{proof}
We first prove the assertion on the kernel of $\mathcal M_u$.  Let $h$ be such that $\mathcal M_u(h)=0$.  Suppose, for contradiction, that $\mathcal B_u(h)\ne0$.  Replacing $h$ by $-h$ if necessary, we may assume
\[
        \mathcal B_u(h)<0.
\]
By Lemma~\ref{lem:calM-nonzero}, choose $g$ such that $\mathcal M_u(g)\ne0$.

For small parameters \(\varepsilon\) and \(\delta\), set
\[
        z_{\varepsilon,\delta}:=u+\varepsilon h+\delta g .
\]
We first project \(z_{\varepsilon,\delta}\) onto the constraint \(K=0\).
Define
\[
        \mathscr K(\tau,\varepsilon,\delta)
        :=
        K(\tau z_{\varepsilon,\delta}).
\]
Then
\[
        \mathscr K(1,0,0)=K(u)=0.
\]
Moreover,
\[
        \partial_\tau \mathscr K(1,0,0)
        =
        2A(u)-2_*\frac d{d+2}P(u).
\]
Since \(K(u)=0\), we have $A(u)=\frac d{d+2}P(u)$ and consequently
\[
        \partial_\tau \mathscr K(1,0,0)
        =
        (2-2_*)A(u)
        =
        -\frac4d A(u)\neq0.
\]
By the implicit function theorem, there exist a neighbourhood
\(\mathcal U\) of \((0,0)\) in \(\mathbb R^2\) and a unique \(C^1\) function
\(
        (\varepsilon,\delta)\mapsto \tau_{\varepsilon,\delta}
\)
defined on \(\mathcal U\) with
\(
        \tau_{0,0}=1,
\)
such that
\[
        K(\tau_{\varepsilon,\delta} z_{\varepsilon,\delta})=0.
\]
We define
\(
        W_{\varepsilon,\delta}
        :=
        \tau_{\varepsilon,\delta} z_{\varepsilon,\delta}.
\)
By definition,
\(
        K(W_{\varepsilon,\delta})=0
\)
for all \((\varepsilon,\delta)\in\mathcal U\).

The first-order expansion of \(\tau_{\varepsilon,\delta}\) follows by differentiating
\[
        \mathscr K(\tau_{\varepsilon,\delta},\varepsilon,\delta)=0
\]
at \((\varepsilon,\delta)=(0,0)\).  Since
\[
        \partial_\varepsilon \mathscr K(1,0,0)=K'(u)[h],
        \qquad
        \partial_\delta \mathscr K(1,0,0)=K'(u)[g],
\]
we obtain
\[
        \partial_\varepsilon\tau_{0,0}
        =
        \frac{d}{4A(u)}K'(u)[h],
        \qquad
        \partial_\delta\tau_{0,0}
        =
        \frac{d}{4A(u)}K'(u)[g].
\]
Consequently,
\[
        \tau_{\varepsilon,\delta}
        =
        1
        +
        \frac{d}{4A(u)}\bigl(\varepsilon K'(u)[h]+\delta K'(u)[g]\bigr)
        +
        O(\varepsilon^2+\delta^2+|\varepsilon\delta|).
\]
Using this expansion together with
\[
        M(z_{\varepsilon,\delta})
        =
        c
        +
        2\varepsilon\int uh
        +
        2\delta\int ug
        +
        O(\varepsilon^2+\delta^2+|\varepsilon\delta|),
\]
we get
\begin{equation}\label{eq:two-parameter-mass-expansion}
        M(W_{\varepsilon,\delta})
        =
        c
        +
        2\varepsilon\mathcal M_u(h)
        +
        2\delta\mathcal M_u(g)
        +
        O(\varepsilon^2+\delta^2+|\varepsilon\delta|).
\end{equation}
Similarly,
\begin{equation}\label{eq:two-parameter-B-expansion}
        B(W_{\varepsilon,\delta})
        =
        B(u)
        +
        2\varepsilon\mathcal B_u(h)
        +
        2\delta\mathcal B_u(g)
        +
        O(\varepsilon^2+\delta^2+|\varepsilon\delta|).
\end{equation}

We now impose the mass constraint.  Define
\[
        \Phi(\varepsilon,\delta)
        :=
        M(W_{\varepsilon,\delta})-c.
\]
Then
\(
        \Phi(0,0)=0.
\)
By \eqref{eq:two-parameter-mass-expansion},
\(
        \partial_\delta\Phi(0,0)
        =
        2\mathcal M_u(g).
\)
Since
\(
        \mathcal M_u(g)\neq0,
\)
the implicit function theorem applies once more and consequently, there exist
\(\varepsilon_1>0\) and a unique \(C^1\) function
\[
        \delta=\delta(\varepsilon),
        \qquad
        |\varepsilon|<\varepsilon_1,
\]
such that
\[
        \delta(0)=0,
        \qquad
        \Phi(\varepsilon,\delta(\varepsilon))=0
\]
and hence also
\(
        M(W_{\varepsilon,\delta(\varepsilon)})=c.
\)

It remains to justify the sharper estimate
\[
        \delta(\varepsilon)=O(\varepsilon^2).
\]
First, differentiating the identity
\(
        \Phi(\varepsilon,\delta(\varepsilon))=0
\)
at \(\varepsilon=0\) gives
\[
        0
        =
        \partial_\varepsilon\Phi(0,0)
        +
        \partial_\delta\Phi(0,0)\delta'(0).
\]
Since
\[
        \partial_\varepsilon\Phi(0,0)=2\mathcal M_u(h)=0
\]
and
\[
        \partial_\delta\Phi(0,0)=2\mathcal M_u(g)\neq0,
\]
we obtain
\(
        \delta'(0)=0.
\)
Combining Taylor expansion, this yields
\(
        \delta(\varepsilon)=o(\varepsilon).
\)
Now substitute \(\delta=\delta(\varepsilon)\) into
\eqref{eq:two-parameter-mass-expansion}.  Since
\(\mathcal M_u(h)=0\), we have
\[
        0
        =
        \Phi(\varepsilon,\delta(\varepsilon))
        =
        2\mathcal M_u(g)\delta(\varepsilon)
        +
        O\bigl(\varepsilon^2+\delta(\varepsilon)^2
        +|\varepsilon\delta(\varepsilon)|\bigr).
\]
Thus, for \(|\varepsilon|\) sufficiently small,
\[
        |\mathcal M_u(g)|\,|\delta(\varepsilon)|
        \le
        C\varepsilon^2
        +
        C\bigl(|\delta(\varepsilon)|+|\varepsilon|\bigr)
        |\delta(\varepsilon)|.
\]
Because \(\delta(\varepsilon)\to0\), we may reduce \(\varepsilon_1\) so that
\[
        C\bigl(|\delta(\varepsilon)|+|\varepsilon|\bigr)
        \le
        \frac12|\mathcal M_u(g)|
        \qquad
        \text{for }|\varepsilon|<\varepsilon_1.
\]
Absorbing the last term into the left-hand side yields
\[
        |\delta(\varepsilon)|
        \le
        C\varepsilon^2.
\]
Hence
\[
        \delta(\varepsilon)=O(\varepsilon^2).
\]

Finally, inserting this estimate into \eqref{eq:two-parameter-B-expansion}, we obtain
\[
        B(W_{\varepsilon,\delta(\varepsilon)})
        =
        B(u)
        +
        2\varepsilon\mathcal B_u(h)
        +
        O(\varepsilon^2).
\]
Since \(\mathcal B_u(h)<0\), it follows that
\[
        B(W_{\varepsilon,\delta(\varepsilon)})<B(u)
\]
for all sufficiently small \(\varepsilon>0\). Since $K(W_{\varepsilon,\delta(\varepsilon)})=0$ and $M(W_{\varepsilon,\delta(\varepsilon)})=c$, this function belongs to $V(c)$.  By \eqref{eq:energy-reduction}, it has energy strictly smaller than $E(u)=m_c$, a contradiction.  Thus $\mathcal B_u=0$ on $\ker\mathcal M_u$.

We next derive the proportionality relation \eqref{eq:calB-ell-calM}.  Recall first that both
\[
        h\mapsto \mathcal M_u(h),
        \qquad
        h\mapsto \mathcal B_u(h)
\]
are real linear functionals on the vector space
\(C_c^\infty(\mathbb R^d\times\mathbb T)\).  By
Lemma~\ref{lem:calM-nonzero}, the functional \(\mathcal M_u\) is not identically
zero.  Hence we may choose
\(
        h_*\in C_c^\infty(\mathbb R^d\times\mathbb T)
\)
such that
\(
        \mathcal M_u(h_*)\neq0.
\)
Define
\[
        \ell:=
        \frac{\mathcal B_u(h_*)}{\mathcal M_u(h_*)}.
\]
We claim that
\[
        \mathcal B_u(h)=\ell\,\mathcal M_u(h)
        \qquad
        \text{for every }h\in C_c^\infty(\mathbb R^d\times\mathbb T).
\]
Indeed, fix an arbitrary \(h\in C_c^\infty(\mathbb R^d\times\mathbb T)\), and set
\[
        \widetilde h
        :=
        h-\frac{\mathcal M_u(h)}{\mathcal M_u(h_*)}\,h_*.
\]
Then
\[
        \mathcal M_u(\widetilde h)
        =
        \mathcal M_u(h)
        -
        \frac{\mathcal M_u(h)}{\mathcal M_u(h_*)}
        \mathcal M_u(h_*)
        =
        0.
\]
Thus \(\widetilde h\in\ker\mathcal M_u\).  Since we have already proved that
\(\mathcal B_u=0\) on \(\ker\mathcal M_u\), it follows that
\(
        \mathcal B_u(\widetilde h)=0.
\)
Using the definition of \(\widetilde h\), we get
\[
        0
        =
        \mathcal B_u(h)
        -
        \frac{\mathcal M_u(h)}{\mathcal M_u(h_*)}
        \mathcal B_u(h_*).
\]
Therefore
\[
        \mathcal B_u(h)
        =
        \frac{\mathcal B_u(h_*)}{\mathcal M_u(h_*)}
        \mathcal M_u(h)
        =
        \ell\,\mathcal M_u(h).
\]
Since \(h\) was arbitrary, \eqref{eq:calB-ell-calM} holds.

It remains to show \(\ell<0\).  By Lemma~\ref{lem:strict-projected-direction-exists}, there exists
\(h_0\in C_c^\infty(\mathbb R^d\times\mathbb T)\) such that
\[
        \mathcal M_u(h_0)>0,
        \qquad
        \mathcal B_u(h_0)<0.
\]
Using \eqref{eq:calB-ell-calM}, we obtain
\[
        \ell
        =
        \frac{\mathcal B_u(h_0)}{\mathcal M_u(h_0)}
        <
        0.
\]
\end{proof}

\begin{theorem}[Optimizers as standing wave solution]\label{thm:standing-wave}
Let $0<c<2\pi\whM(Q)$, and let $u\in V(c)$ be a nonnegative optimizer of $m_c$.  Then there exists $\theta>0$ and $\omega>0$ such that
\begin{equation}\label{eq:anisotropic-positive}
        -\theta\Delta_xu-\partial_y^2u+\omega u=\theta u^{2_*-1}
\end{equation}
in the sense of distributions.  Consequently, the mass-preserving $x$-rescaling
\[
        U(x,y):=\theta^{d/4}u(\sqrt\theta x,y)
\]
is also an optimizer of $m_c$ and satisfies
\begin{equation}\label{eq:isotropic-positive}
        -\Delta_{x,y}U+\omega U=U^{2_*-1}
        \qquad\text{on }\R^d\times\T.
\end{equation}
In other words, $m_c$ possesses a normalized ground state solving the standing wave equation \eqref{eq:standing-wave-intro} with positive frequency $\omega>0$ .
\end{theorem}

\begin{proof}
Let $A=A(u)$, $B=B(u)$, and let $\ell<0$ be the number from Lemma~\ref{lem:tangent-optimality}.  From
\[
        \mathcal B_u(h)-\ell\mathcal M_u(h)=0
\]
for every test function $h$, we get
\[
        \int\partial_yu\,\partial_yh
        -\ell\int uh
        +\frac{d(B-\ell c)}{4A}K'(u)[h]=0.
\]
Set
\[
        \theta:=\frac{d(B-\ell c)}{2A}.
\]
Since $B>0$, $c>0$, and $\ell<0$, we have $\theta>0$.  Using \eqref{eq:K-prime}, the preceding identity becomes
\[
        \int\partial_yu\,\partial_yh
        -\ell\int uh
        +\theta\int\nabla_xu\cdot\nabla_xh
        -\theta\int u^{2_*-1}h=0.
\]
Therefore
\[
        -\theta\Delta_xu-\partial_y^2u-
        \ell u=\theta u^{2_*-1}.
\]
Thus \eqref{eq:anisotropic-positive} holds with
\(
        \omega:=-\ell>0.
\)

Now define
\[
        U(x,y):=\theta^{d/4}u(\sqrt\theta x,y).
\]
The scaling is exactly the mass-critical $x$-scaling, hence it preserves $M$, $K=0$, and the value of $E$ on $V(c)$ by \eqref{eq:x-scaling-identities}.  Therefore $U$ is again an optimizer of $m_c$.

A direct computation using
\[
        \frac d4(2_*-1)=\frac d4+1
\]
shows that \eqref{eq:anisotropic-positive} transforms into
\[
        -\Delta_{x,y}U+\omega U=U^{2_*-1}
\]
which is exactly \eqref{eq:isotropic-positive}.
\end{proof}

\section{Conclusion}\label{sec:conclusion}
Summarizing, we give in this final section the proof of Theorem~\ref{thm:main}.

\begin{proof}[Proof of Theorem~\ref{thm:main}]
The existence of a nonnegative optimizer follows from Theorem~\ref{thm:existence}.  The positive-frequency standing wave equation follows from Theorem~\ref{thm:standing-wave}.
\end{proof}

\subsubsection*{Acknowledgements}
The author was supported by the NSF grant of Guangdong (No. 2024A1515010497), the QB-Program of Guangdong (No. 2024QN11X141) and the NSF grant of China (No. 12301301).

\subsubsection*{Data availability}
Data sharing is not applicable to this article as no datasets were generated or analysed during the current study.


\begin{thebibliography}{10}

\bibitem{R1T1Scattering}
{\sc Cheng, X., Guo, Z., and Zhao, Z.}
\newblock On scattering for the defocusing quintic nonlinear {S}chr\"{o}dinger
  equation on the two-dimensional cylinder.
\newblock {\em SIAM J. Math. Anal. 52}, 5 (2020), 4185--4237.

\bibitem{DancerSolution}
{\sc Dancer, E.~N.}
\newblock New solutions of equations on {{\(\mathbb{R}^n\)}}.
\newblock {\em Ann. Sc. Norm. Super. Pisa, Cl. Sci., IV. Ser. 30}, 3-4 (2001),
  535--563.

\bibitem{actionvsenergy}
{\sc Dovetta, S., Serra, E., and Tilli, P.}
\newblock Action versus energy ground states in nonlinear {Schr{\"o}dinger}
  equations.
\newblock {\em Math. Ann. 385}, 3-4 (2023), 1545--1576.

\bibitem{ForcellaLuoZhao}
{\sc Forcella, L., Luo, Y., and Zhao, Z.}
\newblock Solitons, scattering and blow-up for the nonlinear
  {{S}chr{\"o}dinger} equation with combined power-type nonlinearities on
  {$\mathbb{R}^d\times\mathbb{T}$}.
\newblock Preprint, {arXiv}:2409.15860 [math.{AP}] (2024), 2024.

\bibitem{Liouville_type}
{\sc Gidas, B., and Spruck, J.}
\newblock Global and local behavior of positive solutions of nonlinear elliptic
  equations.
\newblock {\em Comm. Pure Appl. Math. 34}, 4 (1981), 525--598.

\bibitem{LuoSIMA}
{\sc Hajaiej, H., Luo, Y., and Song, L.}
\newblock On existence and stability results for normalized ground states of
  mass-subcritical biharmonic nonlinear {Schr{\"o}dinger} equation on
  {{\(\mathbb{R}^{d} \times \mathbb{T}^{n}\)}}.
\newblock {\em SIAM J. Math. Anal. 56}, 4 (2024), 4415--4439.

\bibitem{HaniPausader}
{\sc Hani, Z., and Pausader, B.}
\newblock On scattering for the quintic defocusing nonlinear {S}chr\"{o}dinger
  equation on {$\Bbb R\times\Bbb T^2$}.
\newblock {\em Comm. Pure Appl. Math. 67}, 9 (2014), 1466--1542.

\bibitem{waveguide_ref_3}
{\sc Kengne, E., Vaillancourt, R., and Malomed, B.~A.}
\newblock Bose{\textendash}einstein condensates in optical lattices: the
  cubic{\textendash}quintic nonlinear {S}chrödinger equation with a periodic
  potential.
\newblock {\em Journal of Physics B: Atomic, Molecular and Optical Physics 41},
  20 (2008), 205202.

\bibitem{Kwong_uniqueness}
{\sc Kwong, M.~K.}
\newblock Uniqueness of positive solutions of {$\Delta u-u+u^p=0$} in {${\bf
  R}^n$}.
\newblock {\em Arch. Rational Mech. Anal. 105}, 3 (1989), 243--266.

\bibitem{Liouville2}
{\sc Li, Y., and Zhang, L.}
\newblock Liouville-type theorems and {H}arnack-type inequalities for
  semilinear elliptic equations.
\newblock {\em J. Anal. Math. 90\/} (2003), 27--87.

\bibitem{Luo_inter}
{\sc Luo, Y.}
\newblock Normalized ground states and threshold scattering for focusing {NLS}
  on $\mathbb{R}^d\times\mathbb{T}$ via semivirial-free geometry, 2022.

\bibitem{Luo_energy_crit}
{\sc Luo, Y.}
\newblock On long time behavior of the focusing energy-critical {NLS} on
  {{\(\mathbb{R}^d\times\mathbb{T}\)}} via semivirial-vanishing geometry.
\newblock {\em J. Math. Pures Appl. (9) 177\/} (2023), 415--454.

\bibitem{luo2021sharp}
{\sc Luo, Y.}
\newblock Sharp scattering for focusing intercritical {NLS} on high-dimensional
  waveguide manifolds.
\newblock {\em Math. Ann.\/} (May 2023).

\bibitem{Luo_LegendreFenchel}
{\sc Luo, Y.}
\newblock A {Legendre}-{Fenchel} identity for the nonlinear {Schr{\"o}dinger}
  equations on {{\(\mathbb{R}^d \times \mathbb{T}^m\)}}: theory and
  applications.
\newblock {\em J. Geom. Anal. 34}, 10 (2024), 40.
\newblock Id/No 313.

\bibitem{LuoCriticalScattering}
{\sc Luo, Y.}
\newblock Critical scattering for the nonlinear {{S}chr{\"o}dinger} equation on
  waveguide manifolds.
\newblock Preprint, {arXiv}:2506.00442 [math.{AP}] (2025), 2025.

\bibitem{luo2026dancer1}
{\sc Luo, Y.}
\newblock On dancer-type solutions for the {L}ane--{E}mden equation via
  semivirial-vanishing geometry, 2026.

\bibitem{waveguide_ref_1}
{\sc Schneider, T.}
\newblock {\em Nonlinear Optics in Telecommunications}.
\newblock Springer Science \& Business Media, Berlin Heidelberg, 2013.

\bibitem{waveguide_ref_2}
{\sc Snyder, A., and Love, J.}
\newblock {\em Optical Waveguide Theory}.
\newblock Springer Science \& Business Media, Berlin Heidelberg, 2012.

\bibitem{TTVproduct2014}
{\sc Terracini, S., Tzvetkov, N., and Visciglia, N.}
\newblock The nonlinear {S}chr\"{o}dinger equation ground states on product
  spaces.
\newblock {\em Anal. PDE 7}, 1 (2014), 73--96.

\bibitem{TNCommPDE}
{\sc Tzvetkov, N., and Visciglia, N.}
\newblock Small data scattering for the nonlinear {S}chr\"{o}dinger equation on
  product spaces.
\newblock {\em Comm. Partial Differential Equations 37}, 1 (2012), 125--135.

\bibitem{TzvetkovVisciglia2016}
{\sc Tzvetkov, N., and Visciglia, N.}
\newblock Well-posedness and scattering for nonlinear {S}chr\"{o}dinger
  equations on {$\Bbb{R}^d\times\Bbb{T}$} in the energy space.
\newblock {\em Rev. Mat. Iberoam. 32}, 4 (2016), 1163--1188.

\bibitem{weinstein}
{\sc Weinstein, M.~I.}
\newblock Nonlinear {S}chr\"{o}dinger equations and sharp interpolation
  estimates.
\newblock {\em Comm. Math. Phys. 87}, 4 (1982/83), 567--576.

\bibitem{RmT1}
{\sc Zhao, Z.}
\newblock On scattering for the defocusing nonlinear {S}chr\"{o}dinger equation
  on waveguide {$\Bbb R^m\times \Bbb T$} (when {$m = 2,3$}).
\newblock {\em J. Differential Equations 275\/} (2021), 598--637.

\bibitem{ZhaoZheng2021}
{\sc Zhao, Z., and Zheng, J.}
\newblock Long time dynamics for defocusing cubic nonlinear {S}chr\"{o}dinger
  equations on three dimensional product space.
\newblock {\em SIAM J. Math. Anal. 53}, 3 (2021), 3644--3660.

\end{thebibliography}

\end{document}